\documentclass[12pt]{amsart}
\usepackage{graphicx}

\usepackage{amsmath,amssymb,framed,graphicx,latexsym,youngtab,epstopdf,booktabs}
\usepackage{color,fourier}

\usepackage{array}
\usepackage{url}

\newtheorem{theorem}{Theorem}[section]

\newtheorem{proposition}[theorem]{Proposition}

\theoremstyle{definition}

\newtheorem{remark}[theorem]{Remark}

\allowdisplaybreaks

\def\calC{\mathcal{C}}

\def\calI{\mathcal{I}}

\def\calO{\mathcal{O}}
\def\calP{\mathcal{P}}

\def\calR{\mathcal{R}}

\def\calT{\mathcal{T}}

\DeclareMathOperator{\Hom}{Hom}
\DeclareMathOperator{\Der}{Der}
\DeclareMathOperator{\End}{End}

\DeclareMathOperator{\gr}{gr}

\DeclareMathOperator{\APoisson}{\textsl{APoisson}}

\DeclareMathAlphabet{\mathbbold}{U}{bbold}{m}{n}

\def\k{\mathbbold{k}}

\begin{document}

\title[Identities for deformation quantizations of almost Poisson algebras]{Identities for deformation quantizations\\ of almost Poisson algebras}

\author{Vladimir Dotsenko}

\address{ 
Institut de Recherche Math\'ematique Avanc\'ee, UMR 7501, Universit\'e de Strasbourg et CNRS, 7 rue Ren\'e-Descartes, 67000 Strasbourg CEDEX, France}

\email{vdotsenko@unistra.fr}

\date{}

\begin{abstract}
We propose an algebraic viewpoint of the problem of deformation quantization of the so called almost Poisson algebras, which are algebras with a commutative associative product and an antisymmetric bracket which is a bi-derivation but does not necessarily satisfy the Jacobi identity. From that viewpoint, the main result of the paper asserts that, by contrast with Poisson algebras, the only reasonable category of algebras in which almost Poisson algebras can be quantized is isomorphic to the category of almost Poisson algebras itself, and the trivial two-term quantization formula already gives a solution to the quantization problem.
\end{abstract}

\maketitle

\section{Introduction}

This short note discusses commutative associative algebras equipped with an anticommutative operation $a_1,a_2\mapsto \{a_1,a_2\}$ that is a biderivation:
\begin{gather*}
\{a_1a_2,a_3\} = \{a_1,a_3\}a_2 + a_1\{a_2,a_3\} ,\\
\{a_1,a_2a_3\} = \{a_1,a_2\}a_3 + a_2\{a_1,a_3\} 
\end{gather*}
(of course, since this operation is assumed anticommutative, it is enough to impose either of those identities). This class of algebras is known under several different names. To the best of our knowledge, such algebras were introduced independently by Cannas and Weinstein in \cite{MR1747916} under the name ``almost Poisson algebras'' and by Shestakov in \cite{MR1755366} under the name ``general Poisson algebras'' (later changed into ``generic Poisson algebras'' in \cite{MR3298985}); a prototypical example of such algebra is given by the algebra of smooth functions on a manifold $M$ equipped with the product operation and the operation $\{f,g\}=\Lambda(df,dg)$, for some bivector field $\Lambda$. We feel that the word ``almost'' indicates more clearly that these algebras are not necessarily Poisson, so we choose to use it throughout the paper.

A topic that has been attracting the attention of both mathematicians and physicists over the past few decades is deformation quantization of the commutative associative product of functions on a manifold in the direction of a Poisson bivector \cite{MR496157}, and in particular the celebrated theorem of Kontsevich about the existence of associative quantized products for arbitrary Poisson manifolds \cite{MR2062626}. In various questions of mathematical physics, the necessity of studying deformation quantization in directions of not necessarily Poisson bivectors has emerged, see, e.g. \cite{Bakas2014,MR3657740,MR1877309,MR3429384,Szabo:2019uV,MR3846952}. This makes one wonder what replaces associative algebras as the class of algebras in which the quantization takes place, see, for instance, the \texttt{MathOverflow} question of Jim Stasheff \cite{335295}. In this paper, we give an answer to this question. 

Our approach is inspired by the following algebraic viewpoint of deformation quantization. In an unpublished work, Livernet and Loday included the operad of Poisson algebras into a one-parametric family of operads that deforms it into the operad of associative algebras. Markl and Remm in \cite[Th.~5]{MR2225770} explained that deformation quantization can be recast in terms of the Livernet--Loday family of operads. The main feature of the latter family is that it deforms the Poisson operad ``in the direction of the Poisson bracket''; algebraically, it corresponds to studying operads containing the Lie operad and such that the associated graded with respect to the appropriate filtration is isomorphic to the Poisson operad. In fact, it follows from our recent work with Bremner \cite{MR4177579} that the Livernet--Loday family is the \emph{only} nontrivial family with this property, and hence from the algebraic point of view, the fact that Poisson algebras ``should be'' quantized into associative algebras can be made into a precise mathematical statement.

Applying the same approach to almost Poisson algebras, we found a result of a somewhat different flavour: an operad containing the ``almost Lie'' operad and such that the associated graded with respect to the appropriate filtration is isomorphic to the almost Poisson operad is necessarily isomorphic to the almost Poisson operad! As a consequence, the only reasonable algebraically defined category of algebras in which almost Poisson algebras can be quantized is isomorphic to the category of almost Poisson algebras itself. Moreover, the ``trivial'' two-term quantization 
 \[
a_1\star a_2=a_1a_2+\frac{\hbar}{2}\{a_1,a_2\}.
 \]
already gives a solution to the quantization problem. This quantization was previously discussed in the literature, see, e.g. \cite{MR2514853}.

\section{Filtrations of operads}

All vector spaces considered in this paper are defined over a ground field $\k$ of zero characteristic. All operads are assumed reduced ($\calP(0)=0$) and connected ($\calP(1)=\k$). This is a short note, and we shall not overload it with excessive recollections, giving just a short recollection of one less standard notion of a filtration of an operad by powers of an ideal. For further information, we refer the reader to \cite{MR3642294,MR2954392}. 

A \emph{(decreasing) filtration} of an operad $\calP$ is a collection $F^\bullet\calP=\{F^k\calP\}_{k\ge 0}$ of subobjects of $\calP$ such that $F^0\calP=\calO-P$, $F^{k+1}\calP\subset F^k\calP$ for all $k\ge 0$, and for all elements $\alpha\in F^k\calP$, $\alpha_i\in F^{k_i}\calP$, $i=1,\ldots,n$, and any structure map $\gamma$ of the operad $\calP$, we have 
 \[
\gamma(\alpha;\alpha_1,\ldots,\alpha_n)\in F^{k+k_1+\cdots+k_n}\calP.
 \] 
For a filtration, we can of course define the associated graded object
 \[
\gr_F\calP:=\bigoplus_{k\ge 0}F^k\calP/F^{k+1}\calP.
 \]
It is obvious that the operad structure on $\calP$ induces a well defined operad structure on $\gr_F\calP$, and we refer to $\gr_F\calP$ equipped with that structure as the \emph{associated graded operad of $\calP$ with respect to the filtration $F$}.

Examples of filtrations studied in this paper come from ideals. Let $\calI$ be an ideal of $\calP$. We define the filtration $F^\bullet_{\calI}\calP$, to be the filtration where the object
$F^k_{\calI}\calP$ is spanned by operad compositions where elements of $\calI$ are used at least $k$ times; formally, those are the compositions along rooted trees where for at least $k$ of the internal vertices of the tree we use elements of $\calI$. We call this filtration the \emph{filtration of $\calP$ by powers of the ideal $\calI$}, highlighting the fact that it is a direct analogue of filtrations by powers of ideals in commutative algebra.

Filtrations by powers of operadic ideals were previously used in \cite{MR3927168} where it is established that the associated graded of the operad of pre-Lie algebras for the filtration by powers of the ideal generated by the Lie bracket recovers the F-manifold identities of Hertling and Manin; the corresponding deformation quantization problem is discussed in \cite{MR4097911}.

\section{Prelude: the case of Poisson algebras}

Let us begin with the case of Poisson algebras, where all the ingredients of our argument can be seen quite well. 
It is well known that, for the filtration of the associative operad by powers of the ideal generated by the Lie bracket $\{a_1,a_2\}=a_1a_2-a_2a_1$, the associated graded is the Poisson operad. Moreover, it is possible to include the Poisson operad into a flat family of operads containing the associative operad: the following result goes back to unpublished work of Livernet and Loday (see \cite{MR2320771,MR3302959,MR2225770} for a proof).

\begin{proposition}\label{prop:LL}
Let $\calO_t$ be the operad over the ring $\k[t]$ generated by a commutative product $a_1a_2$ and an anticommutative bracket $\{a_1,a_2\}$ that satisfy the identities
\begin{align}
\label{eq:relationsLL}
&
\left\{ \quad
\begin{array}{l}
(a_1a_2)a_3 - a_1(a_2a_3) = t \{a_2,\{a_1,a_3\}\} ,
\\ {}
\{a_1a_2,a_3\} = \{a_1,a_3\}a_2 + a_1\{a_2,a_3\} ,
\\ {}
\{\{a_1,a_2\},a_3\} + \{\{a_2,a_3\},a_1\} + \{\{a_3,a_1\},a_2\} = 0.
\end{array}
\right.
\end{align}
The following results hold:
\begin{itemize}
\item The operad $\calO_t|_{t=0}$ is the Poisson operad, and the operad $\calO_t|_{t=1}$ is isomorphic to the associative operad.
\item The component $\calO_t(n)$ is a free $\k[t]$-module of rank $n!$.
\item If the ground field is quadratically closed, for each $t_0\ne 0$ the operad $\calO_t|_{t=t_0}$ is isomorphic to the associative operad.
\end{itemize}
\end{proposition}

Recall \cite{MR496158} that a star product on a vector space $A$ is a $\k[[\hbar]]$-linear binary product $a,b\mapsto a\star b$ on $A[[\hbar]]$ which is commutative modulo $\hbar$ and associative. Writing for $a,b\in A$
 \[
a\star b=a\cdot_0b+\hbar a\cdot_1b+\hbar^2 a\cdot_2b+\cdots,
 \] 
it is immediate to see that $a\cdot_0b$ is a commutative associative product and $a\cdot_1b-b\cdot_1a$ is a Lie bracket that together with $a\cdot_0b$ makes $A$ into a Poisson algebra, so one can actually talk about star products on Poisson algebras. In \cite{MR2225770}, Markl and Remm established the following result.

\begin{proposition}[{\cite[Th.~5]{MR2225770}}]\label{prop:defQ}
The datum of a star product on a Poisson algebra $A$ is equivalent to the datum of a $\calO_{\hbar^2}$-algebra structure on $A[[\hbar]]$ for which 
\begin{gather}
a\star b+b\star a\equiv ab\pmod{\hbar},\\
a\star b-b\star a\equiv \hbar \{a,b\} \pmod{\hbar^2}.
\end{gather}
\end{proposition}

This proposition suggests how to use operads to classify reasonable algebraic targets for deformation quantization ``in the direction of the Poisson bracket''. The word ``algebraic'' above makes it natural to assume that the situation is the same as in Propositions \ref{prop:LL} and \ref{prop:defQ}: we have an operad that is a flat deformation of the Poisson operad, like in the first of those propositions, and the associated graded operad for the filtration by powers of the ideal generated by the skew-symmetric generator is isomorphic to the Poisson operad, like in the second of those propositions.

\begin{proposition}\label{prop:defPoisson}
Suppose that $\calO$ is an operad generated by one binary operation $\cdot$ and such that the associated graded operad for the filtration by powers of the ideal generated by the operation $\{a_1,a_2\}=a_1\cdot a_2-a_2\cdot a_1$ is isomorphic to the Poisson operad. Then there exists $t_0\in\k$ such that $\calO\cong \calO_t|_{t=t_0}$. In particular, over a quadratically closed field, $\calO$ is isomorphic to either the Poisson operad or the associative operad.  
\end{proposition}

\begin{proof}
Like in \cite{MR2225770}, we use the polarization process: we shall view our operad as an operad generated by the two operations 
 \[
a_1a_2=a_1\cdot a_2+a_2\cdot a_1, \qquad \{a_1,a_2\}=a_1\cdot a_2-a_2\cdot a_1.
 \]
It is a straightforward exercise to write down all feasible candidates for the defining relations of the operad~$\calO$ in terms of these generators. 
First, the expression 
 \[
\{\{a_1,a_2\},a_3\}+\{\{a_2,a_3\},a_1\}+\{\{a_3,a_1\},a_2\}
 \] 
(the ``Jacobiator'') must vanish in the associated graded operad, so it actually must vanish on the nose, since it already belongs to the deepest level of filtration that is available in the space of all ternary operations.
Next, the expression 
 \[
\{a_1a_2,a_3\}-\{a_1,a_3\}a_2-\{a_2,a_3\}a_1
 \] 
must vanish in the associated graded operad and is symmetric in $a_1,a_2$, so we must have 
\begin{equation}
\label{eq:relationsPQ-Leibniz}
\{a_1a_2,a_3\}= \{a_1,a_3\}a_2+\{a_2,a_3\}a_1 + v (\{\{a_1,a_3\},a_2\}+\{\{a_2,a_3\},a_1\})
\end{equation}
for some $v\in\k$. 
Finally, the associator 
 \[
(a_1a_2)a_3 - a_1(a_2a_3)
 \]
 must vanish in the associated graded operad and is skew-symmetric in $a_1,a_3$, so we must have 
\begin{multline}
\label{eq:relationsPQ-assoc}
(a_1a_2)a_3 - a_1(a_2a_3)= \\
s (\{a_1,a_2\}a_3-\{a_3,a_2\}a_1) 
+t (\{\{a_1,a_2\},a_3\}-\{\{a_3,a_2\},a_1\}) 
+u \{\{a_1,a_3\},a_2\}=\\
 s (\{a_1,a_2\}a_3-\{a_3,a_2\}a_1) 
+(t+u) \{\{a_1,a_3\},a_2\}
\end{multline}
for some $s,t,u\in\k$ (the second equality uses the Jacobi identity which we already established to hold). Moreover, if we denote by $\sigma$ the cycle $(1 2 3)\in S_3$, the element $1+\sigma+\sigma^2\in\k S_3$ annihilates $(a_1a_2)a_3 - a_1(a_2a_3)$ and $\{\{a_1,a_3\},a_2\}$, so we must have $s=0$ (if not, there will be an additional relation in the associated graded operad). 
Now from the result of \cite[Th.~4.5]{MR4177579} it follows easily that our operad is one of the operads of the Livernet--Loday family. The last assertion is obvious from the results quoted before.
\end{proof}

In plain words, this result means that, if the ground field is sufficiently large, the operad theory can be used to rigorously prove, without any extrinsic assumptions, that there are only two algebraic contexts in which it makes sense to deform Poisson algebras in the direction of the bracket: within the category of Poisson algebras itself and the deformation into associative algebras. (And if some square roots do not exist in $\k$, one can exhibit other meaningful examples of algebras into which Poisson algebras may be quantized). 

\section{The case of almost Poisson algebras}

We shall now apply the same line of thought to seach for algebraic identities that may hold for deformation quantizations of almost Poisson algebras. Let us start with a small remark that should serve as a useful warning. Let us take the operad $\calO$ controlling alternative algebras, and examine the associated graded operad for the filtration by powers of the ideal generated by the operation $\{a_1,a_2\}=a_1\cdot a_2-a_2\cdot a_1$. Then, in terms of the polarized generators $a_1a_2=a_1\cdot a_2+a_2\cdot a_1$ and $\{a_1,a_2\}=a_1\cdot a_2-a_2\cdot a_1$, the alternative identities imply the identities $(a_1a_2)a_3-a_1(a_2a_3)=-1/3 J(a_1,a_2,a_3)-\{\{a_1,a_3\},a_2\}$, $\{a_1a_2,a_3\}=a_1\{a_2,a_3\}+a_2\{a_1,a_3\}$, so there is a map from the almost Poisson operad to the associated graded operad. However, it is easy to see that this map is not an isomorphism; for example, the operation $\{a_1,a_2\}$ satisfies the Malcev identity. 

We shall now establish the following result, which exhibits a major qualitative difference from that of Proposition \ref{prop:defPoisson}. 

\begin{theorem}\label{th:main}
Suppose that $\calO$ is an operad generated by one binary operation $\cdot$, and such that the associated graded operad for the filtration by powers of the ideal generated by the operation $\{a_1,a_2\}=a_1\cdot a_2-a_2\cdot a_1$ is isomorphic to the almost Poisson operad. Then $\calO$ is isomorphic to the almost Poisson operad.  
\end{theorem}

\begin{proof}
Once again, it is straightforward to write down all feasible candidates for the defining relations of the operad $\calO$ in terms of the polarized generators $a_1a_2=a_1\cdot a_2+a_2\cdot a_1$ and $\{a_1,a_2\}=a_1\cdot a_2-a_2\cdot a_1$. 
As above, the expression 
 \[
\{a_1a_2,a_3\}-\{a_1,a_3\}a_2-\{a_2,a_3\}a_1
 \] 
must vanish in the associated graded operad and is symmetric in $a_1,a_2$, so we must have 
\begin{equation}
\label{eq:relationsGPQ-Leibniz}
\{a_1a_2,a_3\}= \{a_1,a_3\}a_2+\{a_2,a_3\}a_1 + v (\{\{a_1,a_3\},a_2\}+\{\{a_2,a_3\},a_1\})
\end{equation}
for some $v\in\k$. 
Similarly, the associator 
 \[
(a_1a_2)a_3 - a_1(a_2a_3)
 \]
 must vanish in the associated graded operad and is skew-symmetric in $a_1,a_3$, so we must have 
\begin{multline}
\label{eq:relationsGPQ-assoc}
(a_1a_2)a_3 - a_1(a_2a_3)= \\
s (\{a_1,a_2\}a_3-\{a_3,a_2\}a_1) 
+t (\{\{a_1,a_2\},a_3\}-\{\{a_3,a_2\},a_1\}) 
+u \{\{a_1,a_3\},a_2\} 
\end{multline}
for some $s,t,u\in\k$. Moreover, if we denote by $\sigma$ the cycle $(1 2 3)\in S_3$, the element $1+\sigma+\sigma^2\in\k S_3$ annihilates $(a_1a_2)a_3 - a_1(a_2a_3)$, so we must have $s=0$ and $u=2t$ (if not, there will be an additional relation in the associated graded operad); note that the Jacobi identity is no longer applicable, so the upshot of our analysis is different.

To summarize, the relations of the operad $\calO$ are
\begin{gather}
(a_1a_2)a_3 - a_1(a_2a_3)= 
t (\{\{a_1,a_2\},a_3\}-\{\{a_3,a_2\},a_1\} +2 \{\{a_1,a_3\},a_2\}),\label{eq:def1}\\ 
\{a_1a_2,a_3\}= \{a_1,a_3\}a_2+\{a_2,a_3\}a_1 + v (\{\{a_1,a_3\},a_2\}+\{\{a_2,a_3\},a_1\})\label{eq:def2}
\end{gather}
for some $t,v\in\k$.

It is well known \cite{MR1747916,MR3298985,MR1755366} that the free almost Poisson algebra is the symmetric algebra on the free anticommutative nonassociative algebra (an ``almost Lie algebra'' in the terminology of \cite{MR1747916}). Operadically, this means that the identities of almost Poisson algebras define a distributive law \cite{MR2954392} between the operads of commutative associative algebras and anticommutative nonassociative algebras. Of course, this imples some constraints before taking the associated graded operad. Formally, the relations of the operad $\calO$ must define a filtered (also called inhomogeneous) distributive law \cite{MR4177579,MR3302959} between those operads. Let us explain what it means concretely. 

As the first step, let us compute the element $\{a_1a_2,a_3a_4\}$ in two different ways using Relations \eqref{eq:def1} and \eqref{eq:def2}. We begin with saying that Relation\eqref{eq:def2} implies   
\begin{multline*}
\{a_1a_2,a_3a_4\}=\\\{a_1,a_3a_4\}a_2+\{a_2,a_3a_4\}a_1 + v (\{\{a_1,a_3a_4\},a_2\}+\{\{a_2,a_3a_4\},a_1\})
\end{multline*}
 and 
\begin{multline*}
\{a_1a_2,a_3a_4\}=\\\{a_1a_2,a_3\}a_4+\{a_1a_2,a_4\}a_3 + v (\{\{a_1a_2,a_3\},a_4\}+\{\{a_1a_2,a_4\},a_3\}).
\end{multline*}
Using Relation \eqref{eq:def2} further, we obtain
\begin{multline*}
\{a_1,a_3a_4\}a_2+\{a_2,a_3a_4\}a_1 + v (\{\{a_1,a_3a_4\},a_2\}+\{\{a_2,a_3a_4\},a_1\})=\\
(\{a_1,a_3\}a_4+\{a_1,a_4\}a_3 + v (\{\{a_1,a_3\},a_4\}+\{\{a_1,a_4\},a_3\}))a_2\\
+(\{a_2,a_3\}a_4+\{a_2,a_4\}a_3 + v (\{\{a_2,a_3\},a_4\}+\{\{a_2,a_4\},a_3\}))a_1\\
+ v (\{\{a_1,a_3a_4\},a_2\}+\{\{a_2,a_3a_4\},a_1\})
\end{multline*}
and
\begin{multline*}
\{a_1a_2,a_3\}a_4+\{a_1a_2,a_4\}a_3 + v (\{\{a_1a_2,a_3\},a_4\}+\{\{a_1a_2,a_4\},a_3\})=\\
(\{a_1,a_3\}a_2+\{a_2,a_3\}a_1 + v (\{\{a_1,a_3\},a_2\}+\{\{a_2,a_3\},a_1\}))a_4\\
+(\{a_1,a_4\}a_2+\{a_2,a_4\}a_1 + v (\{\{a_1,a_4\},a_2\}+\{\{a_2,a_4\},a_1\}))a_3\\
+v (\{\{a_1a_2,a_3\},a_4\}+\{\{a_1a_2,a_4\},a_3\}).
\end{multline*}

Let us now do the calculations modulo the third power of the ideal generated by the bracket; we shall indicate such calculations by the sign $\equiv$ as opposed to $=$. Using Relations \eqref{eq:def1} and \eqref{eq:def2}, we obtain  
\begin{multline*}
(\{a_1,a_3\}a_4+\{a_1,a_4\}a_3 + v (\{\{a_1,a_3\},a_4\}+\{\{a_1,a_4\},a_3\}))a_2\\
+(\{a_2,a_3\}a_4+\{a_2,a_4\}a_3 + v (\{\{a_2,a_3\},a_4\}+\{\{a_2,a_4\},a_3\}))a_1\\
+ v (\{\{a_1,a_3a_4\},a_2\}+\{\{a_2,a_3a_4\},a_1\})\\
\equiv \{a_1,a_3\}(a_2a_4)+\{a_1,a_4\}(a_2a_3)+ v (\{\{a_1,a_3\},a_4\}+\{\{a_1,a_4\},a_3\})a_2\\
+\{a_2,a_3\}(a_1a_4)+\{a_2,a_4\}(a_1a_3) + v (\{\{a_2,a_3\},a_4\}+\{\{a_2,a_4\},a_3\})a_1\\
+ v (\{\{a_1,a_3\}a_4+\{a_1,a_4\}a_3,a_2\}+\{\{a_2,a_3\}a_4+\{a_2,a_4\}a_3,a_1\})\\
\equiv \{a_1,a_3\}(a_2a_4)+\{a_1,a_4\}(a_2a_3)+\{a_2,a_3\}(a_1a_4)+\{a_2,a_4\}(a_1a_3)\\
+v((\{\{a_1,a_3\},a_4\}+\{\{a_1,a_4\},a_3\})a_2+(\{\{a_2,a_3\},a_4\}+\{\{a_2,a_4\},a_3\})a_1)\\
+v((\{\{a_1,a_3\},a_2\}+\{\{a_2,a_3\},a_1\})a_4+(\{\{a_1,a_4\},a_2\}+\{a_2,a_4\},a_1\})a_3)\\
+v (\{a_1,a_3\}\{a_4,a_2\}+\{a_1,a_4\}\{a_3,a_2\}+\{a_2,a_3\}\{a_4,a_1\}+\{a_2,a_4\}\{a_3,a_1\}).
\end{multline*}
Similarly, using the same Relations \eqref{eq:def1} and \eqref{eq:def2}, we obtain 
\begin{multline*}
(\{a_1,a_3\}a_2+\{a_2,a_3\}a_1 + v (\{\{a_1,a_3\},a_2\}+\{\{a_2,a_3\},a_1\}))a_4\\
+(\{a_1,a_4\}a_2+\{a_2,a_4\}a_1 + v (\{\{a_1,a_4\},a_2\}+\{\{a_2,a_4\},a_1\}))a_3\\
+v (\{\{a_1a_2,a_3\},a_4\}+\{\{a_1a_2,a_4\},a_3\})\\
\equiv \{a_1,a_3\}(a_2a_4)+\{a_2,a_3\}(a_1a_4) + v (\{\{a_1,a_3\},a_2\}+\{\{a_2,a_3\},a_1\}))a_4\\
+\{a_1,a_4\}(a_2a_3)+\{a_2,a_4\}(a_1a_3) + v (\{\{a_1,a_4\},a_2\}+\{\{a_2,a_4\},a_1\}))a_3\\
+v (\{\{a_1,a_3\}a_2 + a_1\{a_2,a_3\} ,a_4\} + \{\{a_1,a_4\}a_2 + a_1\{a_2,a_4\} ,a_3\})\\
\equiv \{a_1,a_3\}(a_2a_4)+\{a_1,a_4\}(a_2a_3)+\{a_2,a_3\}(a_1a_4)+\{a_2,a_4\}(a_1a_3)\\
+v ((\{\{a_1,a_3\},a_2\}+\{\{a_2,a_3\},a_1\})a_4 + (\{\{a_1,a_4\},a_2\}+\{\{a_2,a_4\},a_1\})a_3)\\
+v ((\{\{a_1,a_3\},a_4\}+\{\{a_1,a_4\},a_3\})a_2 + (\{\{a_2,a_3\},a_4\}+\{\{a_2,a_4\},a_3\})a_1)\\
+v(\{a_1,a_3\}\{a_2,a_4\}+\{a_1,a_4\}\{a_2,a_3\}+\{a_1,a_4\}\{a_2,a_3\}+\{a_1,a_3\}\{a_2,a_4\}).
\end{multline*}
Comparing the two results, we see that the only difference is in the opposite signs of $\{a_1,a_3\}\{a_2,a_4\}+\{a_1,a_4\}\{a_2,a_3\}$. However, the elements $\{a_1,a_3\}\{a_2,a_4\}$ and $\{a_1,a_4\}\{a_2,a_3\}$ are linearly independent in the associated graded operad, that is modulo the third power of the ideal generated by the bracket. 
Hence we must have $v=0$. 

As the second step, let us compute the element 
 \[
\{(a_1a_2)a_3-a_1(a_2a_3),a_4\}
 \] 
in two different ways using the relations \eqref{eq:def1} and \eqref{eq:def2}, where we already put $v=0$. On the one hand, we have
\begin{multline*} 
\{(a_1a_2)a_3-a_1(a_2a_3),a_4\}=\\
t\left(\{\{\{a_1,a_2\},a_3\}-\{\{a_3,a_2\},a_1\} +2 \{\{a_1,a_3\},a_2\},a_4\}\right) .
\end{multline*}
On the other hand, we have 
\begin{multline*}
\{(a_1a_2)a_3-a_1(a_2a_3),a_4\}=\\
(a_1a_2)\{a_3,a_4\}+\{a_1a_2,a_4\}a_3-(a_2a_3)\{a_1,a_4\}-\{a_2a_3,a_4\}a_1\\
=(a_1a_2)\{a_3,a_4\}+(\{a_1,a_4\}a_2+a_1\{a_2,a_4\})a_3\\-(a_2a_3)\{a_1,a_4\}-(a_2\{a_3,a_4\}+\{a_2,a_4\}a_3)a_1\\
=t (\{\{a_1,a_2\},\{a_3,a_4\}\}-\{\{a_3,a_4\},a_2\},a_1\} +2 \{\{a_1,\{a_3,a_4\}\},a_2\})\\
+
t (\{\{\{a_1,a_4\},a_2\},a_3\}-\{\{a_3,a_2\},\{a_1,a_4\}\} +2 \{\{\{a_1,a_4\},a_3\},a_2\})\\+
t (\{\{a_1,\{a_2,a_4\}\},a_3\}-\{\{a_3,\{a_2,a_4\}\},a_1\} +2 \{\{a_1,a_3\},\{a_2,a_4\}\}),
\end{multline*}
and examining (for example) the coefficient of $\{\{a_1,a_2\},\{a_3,a_4\}\}$, we find $t=0$. The statement of the theorem follows.
\end{proof}

\section{Almost Poisson algebras and Kokoris algebras}

Let us conclude our analysis by discussing the identities of almost Poisson algebras in terms of the ``depolarized'' operation
 \[
a_1\cdot a_2=a_1a_2+\{a_1,a_2\}.
 \]
Recall that a Kokoris algebra \cite{MR116039,shestakov2022nonmatrix}
is an algebra with one binary operation $a_1,a_2\mapsto a_1\cdot a_2$ satisfying the identity
\begin{gather*}
J(a_1,a_2,a_3)=4(a_1,a_2,a_3)-[[a_1,a_3],a_2],
\end{gather*}
where $[a_1,a_2]=a_1\cdot a_2-a_2\cdot a_1$, $J(a_1,a_2,a_3)$ is the Jacobiator of the bracket $[a_1,a_2]$, and $(a_1,a_2,a_3)$ is the associator of the product $a_1\cdot a_2$. 

\begin{proposition}
Depolarization $a_1\cdot a_2=a_1a_2+\{a_1,a_2\}$ induces an isomorphism between the operad of almost Poisson algebras and the operad of Kokoris algebras.
\end{proposition}

\begin{proof}
Direct computation, which essentially is already between the lines of \cite{MR116039,shestakov2022nonmatrix}.
\end{proof}

Our main result, Theorem \ref{th:main}, essentially implies that every almost Poisson algebra admits a trivial quantization 
 \[
a_1\star a_2=a_1a_2+\frac{\hbar}2\{a_1,a_2\}.
 \]
Replacing the operation $\{a_1,a_2\}$ by the operation $\frac{\hbar}2\{a_1,a_2\}$ does not affect the identities between the operations, and so it follows from the above proposition that the star product $a_1\star a_2$ defines a Kokoris algebra. The axioms of Kokoris algebras clearly imply the flexible identity $(a,b,a)=0$, which had been previously noted in \cite{MR3846952}. However, the Kokoris algebra axioms are clearly not enough to ensure alternativity of the star product, which gives a conceptual explanations of the result of \cite{MR3657740} stating that a particular class of nonassociative star products, the ``monopole star products'' are not alternative.

\begin{remark}
The defining identity of Kokoris algebras can be rewritten in a more memorable way using the operation $\{a_1,a_2\}'=\frac12(a_1a_2-a_2a_1)$:
 \[
(a_1,a_2,a_3)=(a_1,a_2,a_3)_{\{-,-\}'},
 \]
that is, the associators of the product and of the bracket $\{a_1,a_2\}'$ are equal.
\end{remark}

\section*{Funding} The author was supported by Institut Universitaire de France and by FAPESP (grant 2022/10933-3).

\section*{Acknowledgements} 
I am grateful to Ivan Pavlovich Shestakov for discussions of Kokoris algebras and general encouragement, and to Dmitry Vassilevich for some inspiring conversations.
The first draft of this note was completed during the author's visit of Tashkent by invitation of Farkhod Eshmatov, and the author is grateful to him for warm hospitality.

\bibliographystyle{plain}
\bibliography{biblio}

\begin{thebibliography}{10}

\bibitem{Bakas2014}
Ioannis Bakas and Dieter L\"ust.
\newblock 3-cocycles, non-associative star-products and the magnetic paradigm
  of r-flux string vacua.
\newblock {\em Journal of High Energy Physics}, 2014(1):171, 2014.

\bibitem{MR496157}
F.~Bayen, M.~Flato, C.~Fronsdal, A.~Lichnerowicz, and D.~Sternheimer.
\newblock Deformation theory and quantization. {I}. {D}eformations of
  symplectic structures.
\newblock {\em Ann. Physics}, 111(1):61--110, 1978.

\bibitem{MR496158}
F.~Bayen, M.~Flato, C.~Fronsdal, A.~Lichnerowicz, and D.~Sternheimer.
\newblock Deformation theory and quantization. {II}. {P}hysical applications.
\newblock {\em Ann. Physics}, 111(1):111--151, 1978.

\bibitem{MR3657740}
Martin Bojowald, Suddhasattwa Brahma, Umut B\"{u}y\"{u}k\c{c}am, and Thomas
  Strobl.
\newblock Monopole star products are non-alternative.
\newblock {\em J. High Energy Phys.}, (4):028, front matter + 17, 2017.

\bibitem{MR4177579}
Murray Bremner and Vladimir Dotsenko.
\newblock Distributive laws between the operads {L}ie and {C}om.
\newblock {\em Internat. J. Algebra Comput.}, 30(8):1565--1576, 2020.

\bibitem{MR3642294}
Murray~R. Bremner and Vladimir Dotsenko.
\newblock {\em Algebraic operads: an algorithmic companion}.
\newblock CRC Press, Boca Raton, FL, 2016.

\bibitem{MR1747916}
Ana Cannas~da Silva and Alan Weinstein.
\newblock {\em Geometric models for noncommutative algebras}, volume~10 of {\em
  Berkeley Mathematics Lecture Notes}.
\newblock American Mathematical Society, Providence, RI; Berkeley Center for
  Pure and Applied Mathematics, Berkeley, CA, 1999.

\bibitem{MR1877309}
Lorenzo Cornalba and Ricardo Schiappa.
\newblock Nonassociative star product deformations for {D}-brane world-volumes
  in curved backgrounds.
\newblock {\em Comm. Math. Phys.}, 225(1):33--66, 2002.

\bibitem{MR2320771}
Vladimir Dotsenko.
\newblock An operadic approach to deformation quantization of compatible
  {P}oisson brackets. {I}.
\newblock {\em J. Gen. Lie Theory Appl.}, 1(2):107--115, 2007.

\bibitem{MR3927168}
Vladimir Dotsenko.
\newblock Algebraic structures of {$F$}-manifolds via pre-{L}ie algebras.
\newblock {\em Ann. Mat. Pura Appl. (4)}, 198(2):517--527, 2019.

\bibitem{MR3302959}
Vladimir Dotsenko and James Griffin.
\newblock Cacti and filtered distributive laws.
\newblock {\em Algebr. Geom. Topol.}, 14(6):3185--3225, 2014.

\bibitem{MR2514853}
A.~S. Dzhumadildaev.
\newblock Anti-commutative algebras with skew-symmetric identities.
\newblock {\em J. Algebra Appl.}, 8(2):157--180, 2009.

\bibitem{MR116039}
Louis~A. Kokoris.
\newblock Nodal non-commutative {J}ordan algebras.
\newblock {\em Canadian J. Math.}, 12:488--492, 1960.

\bibitem{MR3298985}
Pavel~S. Kolesnikov, Leonid~G. Makar-Limanov, and Ivan~P. Shestakov.
\newblock The {F}reiheitssatz for generic {P}oisson algebras.
\newblock {\em SIGMA Symmetry Integrability Geom. Methods Appl.}, 10:Paper 115,
  15, 2014.

\bibitem{MR2062626}
Maxim Kontsevich.
\newblock Deformation quantization of {P}oisson manifolds.
\newblock {\em Lett. Math. Phys.}, 66(3):157--216, 2003.

\bibitem{MR3429384}
V.~G. Kupriyanov and D.~V. Vassilevich.
\newblock Nonassociative {W}eyl star products.
\newblock {\em J. High Energy Phys.}, (9):103, front matter+15, 2015.

\bibitem{MR4097911}
Jiefeng Liu, Yunhe Sheng, and Chengming Bai.
\newblock {$F$}-manifold algebras and deformation quantization via pre-{L}ie
  algebras.
\newblock {\em J. Algebra}, 559:467--495, 2020.

\bibitem{MR2954392}
Jean-Louis Loday and Bruno Vallette.
\newblock {\em Algebraic operads}, volume 346 of {\em Grundlehren der
  mathematischen Wissenschaften [Fundamental Principles of Mathematical
  Sciences]}.
\newblock Springer, Heidelberg, 2012.

\bibitem{MR2225770}
M.~Markl and E.~Remm.
\newblock Algebras with one operation including {P}oisson and other
  {L}ie-admissible algebras.
\newblock {\em J. Algebra}, 299(1):171--189, 2006.

\bibitem{MR1755366}
Ivan~P. Shestakov.
\newblock Speciality problem for {M}alcev algebras and {P}oisson {M}alcev
  algebras.
\newblock In {\em Nonassociative algebra and its applications ({S}\~{a}o
  {P}aulo, 1998)}, volume 211 of {\em Lecture Notes in Pure and Appl. Math.},
  pages 365--371. Dekker, New York, 2000.

\bibitem{shestakov2022nonmatrix}
Ivan~P. Shestakov and Vinicius~Souza Bittencourt.
\newblock Nonmatrix varieties of nonassociative algebras, 2022.

\bibitem{335295}
Jim Stasheff.
\newblock Non-associative deformation quantization.
\newblock MathOverflow.
\newblock \url{https://mathoverflow.net/q/335295} (version: 2020-02-28).

\bibitem{Szabo:2019uV}
Richard Szabo.
\newblock {An introduction to nonassociative physics}.
\newblock In {\em Proceedings of Corfu Summer Institute 2018 ``School and
  Workshops on Elementary Particle Physics and Gravity'' {\textemdash}
  PoS(CORFU2018)}, volume 347, page 100, 2019.

\bibitem{MR3846952}
Dmitri Vassilevich and Fernando Martins~Costa Oliveira.
\newblock Nearly associative deformation quantization.
\newblock {\em Lett. Math. Phys.}, 108(10):2293--2301, 2018.

\end{thebibliography}

\end{document}